\documentclass[review]{elsarticle}

\usepackage{lineno,hyperref}
\modulolinenumbers[5]
\usepackage[linesnumbered,ruled,vlined]{algorithm2e}

\usepackage{amsmath,amssymb}
\usepackage{subfigure}
\usepackage{tikz}
\newcommand{\classNP}{\mathcal{NP}}
\newcommand{\classP}{\mathcal{P}}
\usetikzlibrary[backgrounds]
\usepackage{array}
\usepackage{hyperref}

\newcommand{\abs}[1]{\left\lvert#1\right\rvert}
\newcommand{\sbmultcov}{\ensuremath{\textsc{set } \textsc{multicover}}\xspace}

\newcommand{\bb}{\mathbf{b}}

 \newtheorem{theorem}{Theorem}
 \newtheorem{lemma}{Lemma}
 \newtheorem{definition}{Definition}
 
 \newdefinition{remark}{Remark}
 \newproof{proof}{Proof}

\begin{document}

\begin{frontmatter}

\title{ Repeated randomized algorithm
for the Multicovering Problem }

\author[mymainaddress]{Abbass Gorgi}
\ead{abbass.gorgi@gmail.com}
\corref{mycorrespondingauthor1}
\cortext[mycorrespondingauthor1]{Corresponding author}

\author[mysecondaryaddress]{Mourad El Ouali}
\ead{Elouali@math.uni-kiel.de}

\author[mysecondaryaddress]{Anand Srivastav}
\ead{srivastavi@math.uni-kiel.de}

\author[mymainaddress]{Mohamed Hachimi}
\ead{m.hachimi@uiz.ac.ma}

\address[mymainaddress]{Engineering Science Laboratory, University Ibn Zohr, Agadir, Morocco}
\address[mysecondaryaddress]{ Department of Computer Science, Christian Albrechts University, Kiel, Germany}

\begin{abstract}
Let $\mathcal{H}=(V,\mathcal{E})$ be a hypergraph
with maximum edge size $\ell$ and maximum degree $\Delta$. 
For given numbers $b_v\in \mathbb{N}_{\geq 2}$, $v\in V$,
a set multicover in $\mathcal{H}$ is a
set of edges $C \subseteq \mathcal{E}$ such that every vertex $v$ in $V$
belongs to at least $b_v$ edges in $C$.
\sbmultcov is the problem of finding a minimum-cardinality set multicover.
Peleg, Schechtman and Wool conjectured that unless $\classP =\classNP$, for any fixed $\Delta$
and $b:=\min_{v\in V}b_{v}$, no polynomial-time approximation algorithm
for the \sbmultcov problem has an approximation ratio less than
$\delta:=\Delta-b+1$. Hence, it's a challenge to know whether
the problem of \sbmultcov is not
approximable within a ratio of $\beta \delta$ with a constant $\beta<1$.

 This paper proposes a repeated randomized algorithm for the \sbmultcov problem combined with an initial deterministic threshold step. Boosting success by repeated trials, our algorithm yields an approximation ratio of\\
$ \max\left\{ \frac{15}{16}\delta, \left(1- \frac{(b-1)\exp\left(\frac{ 3\delta+1}{8}\right)}{72 \ell}  \right)\delta\right\}$. The crucial fact is not only that our result improves over the approximation ratio presented by Srivastav et al (Algorithmica 2016) for any $\delta\geq 13$, but it's more general since we set no restriction on the parameter $\ell$.\\
Furthermore, we prove that it is NP-hard to approximate the \sbmultcov problem on $\Delta$-regular hypergraphs within a factor of  $(\delta-1-\epsilon)$.
\\
Moreover we show that the integrality gap for the \sbmultcov problem is at least $\frac{\ln_{2}(n+1)}{2b}$, which for constant $b$ is $\Omega(\ln n )$.

\end{abstract}


\end{frontmatter}

%\linenumbers

\section{Introduction}
This work was intended as an attempt to solve approximately the \sbmultcov problem. A nice formulation of this problem may be given by the notion of hypergraphs.

A hypergraph is a pair 
$\mathcal{H}=(V,\mathcal{E})$, where $V$ is a finite set and 
$\mathcal{E}\subseteq 2^V$ is a family of some subsets of $V$. We call the elements of $V$
vertices and the elements of $\mathcal{E}$ (hyper-)edges. Further, let 
$n := |V|$, $m := |{\cal E}|$.
W.l.o.g.\ let the vertices be enumerated as $v_1,v_2,\dots,v_n$ and the edges as $E_1,E_2,\dots,E_m$. As usually the degree of a vertex $v$ (notation $d(v)$) is the number of hyperedges it appears in. Let $\Delta:=\max_{v\in V}d(v)$ be the maximum degree.
 Furthermore, if the degree of every vertex is exactly $\Delta$, then ${\cal H}$ is called $\Delta$-regular. We define the number of vertices of a hyperedge as its size. If the size of all hyperedges is exactly $\ell$, i.e., $\forall E\in \mathcal{E},\, |E|=\ell$, then $\mathcal{H}$ is  $\ell$-uniform.
\noindent Let $\bb:=(b_1,b_2,\dots,b_n) \in \mathbb{N}_{\geq 2}^{n}$ be given.
If a vertex $v_i$, $i\in[n]$, is contained in at least $b_i$ edges
of some subset $C \subseteq \mathcal{E}$,
we say that the vertex $v_i$ is fully covered by $b_i$ edges in $C$.
A set multicover in $\mathcal{H}$ is a set of edges $C \subseteq \mathcal{E}$ 
such that every vertex $v_i$ in $V$ is fully covered by $b_i$ edges in $C$.
The \sbmultcov problem is the task of finding a set multicover of minimum cardinality.

\noindent {\bf Related Work.} 
The set cover problem $(b=1)$ is known to be NP-hard~\cite{Karp} and has been intensively explored for decades. Several deterministic approximation algorithms are exhibited for
this problem~\cite{Ba01,GaKhSr01,Hoch82,Koufogiannakis},
all with approximation ratios $\Delta$.  Furthermore,  Johnson~\cite{Johnson74}
and Lov\'asz~\cite{Lovasz94} gave a greedy algorithm
with performance ratio $H(\ell)$,
where $ H(\ell) =\sum_{i=1}^{\ell}\frac{1}{i}$ is the harmonic number.
Notice that $H(\ell)\leq 1+\ln(\ell)$.
For hypergraphs with bounded $\ell$, Duh and Fürer~\cite{Duh97}
used the technique called semi-local optimization, 
improving $H(\ell)$ to $H(\ell)-\frac{1}{2}$.\\
Unlike the set cover problem, the case $b\geq 2$ of the \sbmultcov problem is less known. Let us give a summary of the known approximability results.
In paper~\cite{RajVaz}, Vazirani using primal-dual schema extended the result of Lov\'asz~\cite{Lovasz94} for $b\geq 1$. Later Fujito et al.~\cite{Fujito} improved the algorithm of Vazirani and achieved an approximation ratio of $H(\ell) -\frac{1}{6}$ for $\ell$ bounded.
Hall and Hochbaum~\cite{HH86} achieved by a greedy algorithm based on LP duality an approximation ratio of $\Delta$.
By a deterministic threshold algorithm
Peleg, Schechtman, and Wool in 1997~\cite{PSW97,PSW93} improved this result and gave an approximation ratio of $\delta$.
They were also the first to propose an approximation algorithm for the \sbmultcov problem with approximation ratio below $\delta$, namely a randomized rounding algorithm with performance ratio $(1-(\frac{c}{n})^\frac{1}{\delta})\cdot\delta$ for a small constant $c>0$.
However, their ratio is depending on $n$, 
and asymptotically tends to $\delta$. Furthermore Peleg, Schechtman and Wool conjectured that for any fixed $\Delta$ and $b:=\min_{i\in [n]}b_i$ the problem cannot be approximated by a ratio smaller than $\delta:= \Delta-b+1$ unless $\mathcal{P}=\mathcal{NP}$. Hence it remained an open problem whether an approximation ratio of $\beta\delta$ with $\beta<1$ constant can be proved.
A randomized algorithm of hybrid type was later given
by Srivastav et al~\cite{EMS16}.
Their algorithm achieves for hypergraphs with
$l \in \mathcal{O}\left(\max\{(nb)^\frac{1}{5},n^\frac{1}{4}\}\right)$ an approximation ratio of
$ \left(1-\frac{11 (\Delta - b)}{72l}\right)\cdot \delta $
with constant probability.\\
Concerning the algorithmic complexity, the \sbmultcov problem has still not been investigated. In contrast to the set cover problem, it is known that the problem is hard to approximate to within $\Delta-1-\epsilon$, unless $\classP =\classNP$~\cite{DinKhotRegev}, and to within $\Delta-\epsilon$ under the UGC~\cite{KR08} for any fixed $\epsilon > 0 $. Unless $\classP =\classNP$ there is no $(1-\epsilon) \ln n$ approximation~\cite{Feige}. This motivated us to study this aspect of the problem.\\
\noindent {\bf Our Results.}  
The main contribution of our paper is the combination
of a deterministic threshold-based algorithm
with repeated randomized rounding steps.
The idea is to algorithmically discard instances
that can be handled deterministically
in favor of instances
for which we obtain a constant-factor approximation less than $\delta$
using a repeated randomized strategy.
 
 Our hybrid randomized algorithm is designed as a cascade of a deterministic and a repeated randomized rounding step followed by greedy repair if the randomized solution is not feasible. First, the relaxed problem of the \sbmultcov problem is solved. The successive actions depend on the cardinality of a set of hyperedges that will be defined according to the relaxed problem output. 
Our algorithm is an extension of an example
given in~\cite{EFS14,EFS14a,EMS16,GaKhSr01,HH86,PSW93}
for the vertex cover, partial vertex cover and \sbmultcov problem
in graphs and hypergraphs.
 
The methods used in this paper rely on an application
of an extension of the Chernoff-Hoeffding bound theorem
for sums of independent random variables
and are based on estimating the variance of the summed random variables
for invoking the Chebychev-Cantelli inequality.
Our algorithm yields a performance ratio of
$ \max\left\{ \frac{15}{16}\delta, \left(1- \frac{(b-1)\exp\left(\frac{ 3\delta+1}{8}\right)}{72 \ell}  \right)\delta\right\}$. 
This ratio means a constant factor of less than $\delta$ for many settings of the parameters $\delta$, $b$, and $\ell$.
It is asymptotically better than the former approximation ratios due to Peleg et al.\ and Srivastav et al.
Furthermore, using a reduction of the set cover problem on $\Delta$-regular hypergraphs to the \sbmultcov problem on $\Delta+b-1$-regular hypergraphs, we show that it is NP-hard to approximate the \sbmultcov problem on $\Delta$-regular hypergraphs within a factor of $(\delta-1-\epsilon)$. Moreover, we show that the integrality gap for the natural LP formulation of the \sbmultcov problem is at least
$\frac{\ln_{2}(n+1)}{2b}$, which for constant $b$ is $\Omega(\ln n )$.
\smallskip 

\renewcommand{\baselinestretch}{1.5}
\begin{table}[h]
\begin{center}
Fundamental results and approximations for \sbmultcov problem \\\vskip 0.2cm
  \begin{tabular}{|c |c | c |  }
\hline
Hypergraph  & Approximation ratio  \\
\hline
- &
 $ H(\ell)$\cite{RajVaz}
 \\\hline
bounded  $\ell$  &
 $H(\ell) -\frac{1}{6}$~\cite{Fujito}
 \\\hline
- &
  $\delta$~\cite{HH86,PSW93}
   \\\hline
- &
$(1-(\frac{c}{n})^\frac{1}{\delta})\cdot \delta$   where $ c>0$ is a constant.~\cite{PSW97}
 \\ \hline
$l \in \mathcal{O}\left(\max\{(nb)^\frac{1}{5},n^\frac{1}{4}\}\right)$
 & $\left(1 - \frac{11(\Delta - b)}{72\ell} \right)\cdot \delta $~\cite{EMS16} 
  \\\hline
 - & $ \max\left\{ \frac{15}{16}\delta, \left(1- \frac{(b-1)\exp\left(\frac{ 3\delta+1}{8}\right)}{72 \ell}  \right)\delta\right\}$ 
  $\ $ (this paper) 
\\ \hline

\end{tabular}
\end{center}
\end{table}
\renewcommand{\baselinestretch}{1.0}

\noindent {\bf Outline of the paper.} In Section 2, we give all the definitions and the tools needed for our analysis. In Section 3, we present a randomized algorithm of hybrid type and its analysis. In Section 4, we give a lower bound for the problem.  In Section 5, we discuss the integrality gap of the LP formulation of the problem.
\section{ Definitions and preliminaries}
For the later analysis we will use the following extension of Chernoff-Hoeffding Bound inequality for a sum of independent random variables. It is often used if one only has a bound on the expectation: 
\begin{theorem}[see %\cite{AV} or 
\cite{Doerr}]\label{Doerr}
Let $X_1,\ldots, X_n$ be independent $\{0,1\}$-random variables.
Let $X= \sum_{i=1}^n X_i$ and suppose $\mathbb{E}(X)< \mu $. For every $ 0<\beta\leq 1$ we have
\begin{equation*} \Pr[X\geq (1+\beta) \mu ]\leq
\exp{\left(-\frac{\beta^2\mu }{3 } \right)}.
\end{equation*}
\end{theorem}
\vskip 0.2cm\noindent
%{\em Concentration Inequalities.}
A further useful concentration theorem we will use is the
Chebychev-Cantelli inequality:
\begin{theorem}[see~\cite{MR}, page 64]\label{Che-Can}
Let $X$ be a non-negative random variable with finite mean $\mathbb{E}(X)$
and variance {\rm Var}$(X)$. Then for any $a>0$ it holds that
\begin{eqnarray*}
 \Pr(X\leq \mathbb{E}(X)-a)&\leq & \frac {{\rm Var} (X)}{{\rm Var}(X)+a^2}\cdot
\end{eqnarray*}
\end{theorem}
Our lower bound proof for the problem relies on extending the following theorem from the case of $b=1$ to the case of $b\geq 2$.
\begin{theorem} [I. Dinur et al, 2005~\cite{DinKhotRegev}]\label{Dinur} 
For every integer $l\geq 3$ and every $\epsilon >0$, it is NP-hard to approximate the minimum vertex cover problem on $\ell$-uniform hypergraphs within a factor of  $(\ell-1-\epsilon)$.\@
\end{theorem}

A key notion of linear programming relaxations is the concept of Integrality Gap.
\begin{definition}
Let $\cal{I}$ be a set of instances, the Integrality Gap for minimization problems is defined as
\begin{equation*}
\sup_{i\in \cal{I}}{\frac{{\rm Opt}(I)}{{\rm Opt}^{*}(I)}}.
\end{equation*}
\end{definition}
\smallskip
\section{ The multi-randomized rounding algorithm }
Let ${\cal H}=(V,{\cal E})$ be a hypergraph with maximum vertex degree $\Delta$ and maximum edge size $\ell$.
An integer linear programming formulation of the \sbmultcov problem is the following:
\begin{eqnarray*}
 &&\min \sum_{j=1}^m x_j,\\
% \vspace{0,2cm}
\mbox{ILP}(\Delta, {\bf b}):\qquad  &&\sum_{j=1}^m a_{ij}x_j\geq b_{i} \quad
 \mbox{ for all } i \in [n],\\
%   \vspace{0,2cm}
  &&x_j\in \{0,1\}\quad\mbox{ for all } j\in [m],
  \end{eqnarray*}
where $ A=(a_{ij})_{i \in [n], \,j \in [m]}\in \{0,1\}^{n \times m}$
is the vertex-edge incidence matrix of ${\cal H}$
and ${\bf b}=(b_1,b_2,\dots,b_n) \in \mathbb{N}_{\geq 2}^{n}$
is the given integer vector.
 For every vertex $v$, we define
 $\Gamma(v):=\{E\in \mathcal{E} \mathrel{|} v\in E \}$ 
the set of edges incident 
to $v$.\\
The linear programming relaxation LP($\Delta,\,{\bf b}$)
of ILP($\Delta,\,{\bf b}$) is given by relaxing the integrality constraints
to  $x_j\in [0,1]$ for all $j \in [m]$.
Let $\mathrm{Opt}$ resp. ${\rm Opt}^{*}$
be the value of an optimal solution to ILP($\Delta,\,{\bf b}$)
resp. LP($\Delta,\,{\bf b}$).
Let $(x^{\ast}_{1},\ldots,x^{\ast}_{m})$ be the optimal solution of the LP($\Delta,\,{\bf b}$).
So ${\rm Opt}^{*} =\sum_{j=1}^mx^{*}_{j}$
and ${\rm Opt}^{*} \leq \mathrm{Opt}$.

The next lemma shows that the $b_i$ greatest values of the LP variables corresponding to the incident edges for any vertex $v_i$ are all greater than or equal to $\frac{1}{\delta}$.
\begin{lemma}[see~\cite{PSW93}]\label{lemma:1}
Let $b_i,d,\Delta, n \in\mathbb{N} $ with $2\leqslant b_i\leqslant d-1\leqslant \Delta -1, i\in [n]$ .  Let $x_j\in [0,1],j\in [d]$,  such that   $\displaystyle  \sum_{j=1}^{d}x_j\geqslant b_i$.
Then at least $b_i$ of the $x_j$ fulfill the inequality $x_j\geqslant \frac{1}{\delta}$.
\end{lemma}
Our second lemma shows that the $b_i-1$ greatest  values of the LP variables corresponding to the incident edges  for any vertex $v_i$ are all  greater than or equal to $\frac{2}{\delta+1}$ and with Lemma~\ref{lemma:1} we take the sum over the $b_i$ greatest  values of the LP variables corresponding to the incident edges  for any vertex $v_i$.
\begin{lemma}\label{lemma:2}
Let $b_i,d,\Delta, n \in\mathbb{N} $ with $2\leqslant b_i\leqslant d-1\leqslant \Delta -1, i\in [n]$ .  Let $x_j\in [0,1],j\in [d]$,  such that   $\displaystyle  \sum_{j=1}^{d}x_j\geqslant b_i$.
Then at least $b_{i}-1$ of the $x_j$ fulfill the inequality $x_j\geqslant \frac{2}{\delta+1}$ and there exists an element $x_j$, distinct to all of them, that fulfills the inequality $x_j\geqslant \frac{1}{\delta}$ .
\end{lemma}

\begin{proof}
W.l.o.g.\ we suppose $x_1\geq x_2\geq \cdots \geq x_{b_i }\geq \cdots\geq x_d$.\\
Hence $ b_i-2\geq   \displaystyle  \sum_{j=1}^{b_i-2}x_j$
and  $ (d -b_i+2) x_{b_{i-1}}\geq   \displaystyle  \sum_{j=b_i-1}^{d}x_j$.\\
Then
\begin{eqnarray*}
b_i-2+ (\Delta -b+2) x_{b_i-1} &\geq &b_i-2+ (\Delta -b_i+2) x_{b_i-1}\\
&\geq &b_i-2+ (d -b_i+2) x_{b_i-1}\\
  &\geq&  \sum_{j=1}^{b_i-2}x_j+  \sum_{j=b_i-1}^{d}x_j 
 =  \displaystyle  \sum_{j=1}^{d}x_j \\
 &\geq & b_i
\end{eqnarray*}
So we have
$x_{b_i-1}\geq \frac{2}{\delta+1}$.\\
Since for all $j\in [b_i-1]\; , \ x_{j} \geq x_{b_i-1 }$
then for all $j\in [b_i-1]\;  ,\;  x_{j}\geq \frac{2}{\delta+1}$.\\
Furthermore, by Lemma $1$  and the assumption on the orders of the  variables $x_j$, for all $j\in [b_i]\;$  we have  $x_{j}\geq \frac{1}{\delta}$
and particularly 
$x_{b_i}\geq \frac{1}{\delta}$.
\end{proof}

\subsection{The algorithm } 
In this section we present an algorithm with conditioned randomized rounding based on the properties satisfied by two generated sets, $C_1$ and $C_2$.

\begin{algorithm}
\label{alg:msetcover1}
\caption{ SET MULTICOVER}
\SetKwInOut{Input}{Input}\SetKwInOut{Output}{Output}
\Input{A hypergraph $\mathcal{H}=(V,\, \mathcal{E})$ with maximum degree $\Delta$ and maximum hyperedge size $\ell$,
numbers $b_i\in \mathbb{N}_{\geq 2} \text{ for }  i\in[n]$, $  b:=\min_{i\in [n]}b_{i}$, $\epsilon \in (0,1)$, a constant $ k\in \mathbb{N}_{\geq 2}$  and
 $\delta=\Delta-b+1$.}
\Output{A  set multicover $C$}
   \begin{enumerate}
     \item Initialize $C:=\emptyset $. Set $\lambda=\frac{\delta+1}{2}\;$, $\alpha=\frac{(b-1)\delta \epsilon^k}{6\ell}\times \exp\left(a_{k,\epsilon}\right)$ with $a_{k,\epsilon}=\frac{k(1-\epsilon) +(\delta-1)(1-\epsilon^a)}{2}$   and $ \lambda_0 = (1-\epsilon) \delta $.
     \item  Obtain an optimal solution $x^* \in [0,1]^m$ by solving the LP($\Delta,\,{\bf b}$)  relaxation.
     \item \text{Set} $C_1:=\{E_j \in \mathcal{E} \mathrel{|} x_j^{\ast}\geq\frac{1}{\lambda}\}$, $\ C_2:=\{E_j \in \mathcal{E} \mathrel{|} \frac{1}{\lambda}> x_j^{\ast}\geq\frac{1}{\delta}\}$ \\\text{and} $C_3:=\{E_j\in \mathcal{E}\mathrel{|} 0<x_j^{\ast}<\frac{1}{\delta}\}$.
\item Take all edges of the set $C_1$ in the cover $C$.
\item if $|C_1|\geq \alpha \cdot \mathrm{Opt}^{*}$  then return $C=C_1\cup C_2$.\\
 Else \text{(Multi-randomized Rounding)}
\begin{enumerate}
	\item  For all edges $E_j\in C_2$  include the edge $E_j$ in the cover $C$,
independently for all such $E_j$, with probability $\lambda_0 x_j^*$, $k$ times.\\
$\left( \text{ If, in any of these $k$ biased coin flips
shows head, include the edge } E_j \text{ in the cover.}\right)$\\
	\item
 For all edges $E_j\in C_3$ include the edge $E_j$ in
      the cover $C$, independently for all such $E_j$,  with probability $(1-\epsilon^k) \delta x_j^*$.
     \item (Repairing) Repair the cover $C$ (if necessary) as follows: Include arbitrary edges from $C_2$, incident to the vertices $v_i$  not fully covered, to $C$ until all vertices  are fully covered.
     \item Return the cover  $C$.
    \end{enumerate}
    \end{enumerate}
\end{algorithm}

\noindent 
In step $2$ we solve the linear programming relaxation LP($\Delta,\,{\bf b}$) in polynomial time, using some known polynomial-time procedure, e.g.\ the interior point method.
Next we take into the cover all edges of the sets $C_1$ resp.\ $C_2$.
Since the LP variable value $x^*_j$ that corresponds to an  edge $E_j$ from  the set $ C_1$  is greater than or equal to $ \frac{2}{\delta+1} $ and the value $x^*_j$ that corresponds to an  edge $E_j$ from  the set $ C_2$  is less than $ \frac{2}{\delta+1} $, we have
\begin{equation}\label{cover}
|C_1|+|C_2|=|C|   \text{\quad  and\quad  } C_1\cap C_2=\emptyset
\end{equation}\\
%%%%%%%%%%%%%%%%%%%%%%%%%%%%%%%%%
%%%%%%%%%%%%%%%%%%%%%%%%%%%%%%%%
\subsection{Analysis of the algorithm}
\textbf{Case $ \mathbf{|C_1|\geq \alpha \cdot \mathrm{Opt}^{*}}$.}
\begin{theorem}\label{theoremcase1}
Let $\mathcal{H}$ be a hypergraph with maximum vertex degree $\Delta$ and maximum edge size $\ell$. Let $\alpha=\frac{(b-1)\delta\epsilon^k }{6 \ell}\times \exp\left(a_{k,\epsilon}\right)$ with $a_{k,\epsilon}=\frac{k(1-\epsilon) +(\delta-1)(1-\epsilon^a)}{2}$ as defined in Algorithm~\ref{alg:msetcover1}. If $|C_1|\geq \alpha \cdot \mathrm{Opt}^{*}$ then Algorithm~\ref{alg:msetcover1} returns a set multicover $C$ such that 
$$|C| < \left(1- \frac{(b-1)\epsilon^k}{18 \ell}\times \exp\left(a_{k,\epsilon}\right) \right)\delta \cdot \mathrm{Opt}^{*} $$
\end{theorem}
\vskip 0.2cm\noindent
\begin{proof} 
The proof is straightforward, using the definitions
of the sets $C_1$ and $C_2$.
\begin{eqnarray*}
 \delta \mathrm{Opt}^{*}= \sum_{j=1}^m\delta x^*_j
 &\geq & \displaystyle  \sum_{E_j\in C_1}\delta x^*_j+ \sum_{E_j\in C_2}\delta x^*_j \\
  &\geq &\frac{2\delta}{\delta+1}| C_1|+|C_2| \\
  &\geq &\frac{2\delta}{\delta+1}| C_1|+\left(|C|-|C_1| \right)\\
 &\geq &\frac{\delta-1}{\delta+1} |C_1|  + |C|\\
   &\overset{\delta \geq 2}{\geq} &\frac{1}{3} |C_1|  + |C|\\
   &\geq &\frac{1}{3}\alpha \cdot \mathrm{Opt}^{*}  + |C|.
\end{eqnarray*}
Hence
$$|C|\leq  \left(1- \frac{(b-1)\epsilon^k}{18 \ell}\times \exp\left(a_{k,\epsilon}\right) \right)\delta \cdot \mathrm{Opt}^{*}$$
\end{proof}
 \textbf{Case $ \mathbf{|C_1|< \alpha \cdot \mathrm{Opt}^{*}}$.}
 
 Let $X_{1},\ldots,X_{m}$ be $\{0,1\}$-random variables defined as follows:
\begin{eqnarray*}
 X_j=
 \begin{cases}
  1 &\text{if the edge}\, E_j \, \text{was picked into the cover
 before repairing}\\
  0 &\text{otherwise}.
 \end{cases}
\end{eqnarray*}
Note that the $X_1,\ldots,X_m$ are independent for a given $x^* \in [0,1]^m$. 
For all  $i\in[n] $ we define the $\{0,1\}$- random variables $Y_{i}$ as follows:
\begin{eqnarray*}
 Y_{i}= 
 \begin{cases}
  1 &\text{if the vertex}~ v_{i}~\text{is fully covered before repairing}\\
  0 &\text{otherwise}.
 \end{cases}
\end{eqnarray*}
 We denote by $X:=\sum_{j=1}^m X_j$ and $Y:=\sum_{i=1}^n Y_i$ the cardinality of the cover and the cardinality of the set of fully covered vertices before the step of repairing, respectively. At this step by Lemma~\ref {lemma:2}, one more edge for each vertex is at most needed to be fully covered. The cover $C$ obtained by Algorithm~\ref{alg:msetcover1} is bounded by
\begin{equation}\label{expection}
\abs C \leq X+n-Y.
\end{equation}
\noindent Our next lemma provides upper bounds on the expectation of the random variable $X$ and the expectation and variance of the random variable $Y$, which we will use to proof Theorem~\ref{Maintheorem}. This is a restriction of Lemma $4$ in~\cite{EMS16} to the last case in Algorithm~\ref{alg:msetcover1}.
\begin{lemma}\label{lemma:random}
Let $l$ and $\Delta$ be the maximum size of an edge and the maximum vertex degree, respectively.
Let $\epsilon \in \left[\frac{\delta-1}{2\delta},\left(\frac{\delta-1}{2\delta}\right)^{\frac{1}{k}}\right]$, $a_{k,\epsilon}=\frac{k(1-\epsilon) +(\delta-1)(1-\epsilon^a)}{2}$,  $ \lambda_0 = (1-\epsilon) \delta $ and $\lambda=\frac{\delta+1}{2}$
as in \mbox{Algorithm~\ref{alg:msetcover1}}. We have
\begin{description}
\item $\mathrm{(i)}$ $\mathbb{E}(Y)\geq (1-\exp\left( -2a_{k,\epsilon} \right))n$.
\smallskip
\item $\mathrm{(ii)}$ ${\rm Var}(Y)\leq 2n^2\exp\left( -2a_{k,\epsilon} \right)$.
\smallskip
\item $\mathrm{(iii)}$
$ \mathbb{E}(X)\leq (1-\epsilon^k)\delta \mathrm{Opt}^{*}$.
\smallskip
\item $\mathrm{(iv)}$
$ \dfrac{(b-1)n}{\alpha \ell} <  \mathrm{Opt}^{*}$.
\end{description}
\end{lemma}
\begin{proof}
(i) Let $i \in [n]$, $r = d(i) - b_i + 1$.
If $\abs{C_1 \cap \Gamma(v_i)} \geq b_i$, then the vertex $v_i$ is fully covered and $\Pr(Y_{i}=0)=0$.
Otherwise we get by Lemma~\ref{lemma:2} that
$\abs{C_1 \cap \Gamma(v_i)} = b_i-1$ and there exists at least one more edge from $C_2$ with $x_j\geq \frac{1}{\delta}$,
so we have $\sum_{E_j \in \Gamma(v_i) \cap C_2} x_j^* \geq \frac{1}{\delta}$
and by the inequality constraints it holds that $\sum_{E_j \in \Gamma(v_i) \cap \left( C_2\cup C_3\right)} x_j^* \geq 1$. Therefore
\begin{eqnarray*}
\Pr(Y_i = 0) &=& \left(\prod_{E_j \in \Gamma(v_i) \cap C_2}(1- \lambda_0 x_j^*) \right)^{k}
\prod_{E_j \in \Gamma(v_i) \cap C_3}\left(1- (1-\epsilon^k)\delta x_j^*\right) 
\\
&=& \prod_{E_j \in \Gamma(v_i) \cap C_2}\left(1- \lambda_0 x_j^*\right) ^{k}
\prod_{E_j \in \Gamma(v_i) \cap C_3}\left(1- (1-\epsilon^k)\delta x_j^*\right)
\\
&\leq &\prod_{E_j \in \Gamma(v_i) \cap C_2} \exp(-k\lambda_0 x_j^* )
\prod_{E_j \in \Gamma(v_i) \cap C_3}\exp(-(1-\epsilon^k)\delta x_j^*)
\\
&=&  \exp\left(-k(1-\epsilon)\delta \sum_{E_j \in \Gamma(v_i) \cap C_2} x_j^* \right)
\cdot
 \exp\left(-(1-\epsilon^k)\delta \sum_{E_j \in \Gamma(v_i) \cap C_3} x_j^* \right)
\\
&=&  \exp\left( \left(-k(1-\epsilon) +(1-\epsilon^k)\right)\delta \sum_{E_j \in \Gamma(v_i) \cap C_2} x_j^* \right)
\cdot
 \exp\left(-(1-\epsilon^k)\delta \sum_{E_j \in \Gamma(v_i) \cap\left(C_2\cup C_3\right)} x_j^* \right).
\end{eqnarray*} 
Since 
$1-\epsilon^k=(1-\epsilon)\sum_{i=0}^{k-1} \epsilon^i \leq k(1-\epsilon) $, we have $-k(1-\epsilon) +1-\epsilon^k \leq 0$.

It follows that
\begin{eqnarray*}
\Pr(Y_i = 0) &\leq &  \exp\left( -k(1-\epsilon) +(1-\epsilon^k) \right)
\cdot
 \exp\left(-(1-\epsilon^k)\delta  \right)
\\
&= &  \exp\left( -2a_{k,\epsilon} \right).
\end{eqnarray*}

Therefore
\begin{align*}
\mathbb{E}(Y) &= \sum_{i=1}^n \Pr(Y_i = 1)
= \sum_{i=1}^n (1 - \Pr(Y_i = 0))\\
&\geq \sum_{i=1}^n (1 - \exp\left(-2a_{k,\epsilon}\right))\\
&\geq (1-\exp\left(-2a_{k,\epsilon}\right))n.
\end{align*}

(ii)
Since
\begin{equation*}
Y=\sum_{i=1}^{n}Y_{i} \leq n,
\end{equation*}
we have
\begin{equation*}
\mathbb{E}(Y^2)\leq n^2.
\end{equation*}
Thus,
\begin{align*}
{\rm Var}(Y) &= \mathbb{E}(Y^2)-\mathbb{E}(Y)^2
\leq  n^2 -(1-\exp\left(-2a_{k,\epsilon}\right))^2n^2 \\
&\leq n^2\left(1-(1-\exp\left(-2a_{k,\epsilon}\right))^2\right)
\\&\leq 
2n^2\exp\left( -2a_{k,\epsilon} \right).
\end{align*}

(iii) Let $E_j$ be an edge from $C_2$. By Lemma~\ref{lemma:2} we have $ \frac{1}{\delta}\leq x^*_j< \frac{2}{\delta+1}$.\\
Recall that we include independently the edge $E_j$ in the cover $C$, with probability $\lambda_0 x_j^*$, $k$ times.
Since $\frac{\delta-1}{2\delta}\leq \epsilon$, we have $  1-\epsilon \leq \lambda_0 x^*_j< \frac{2}{\delta+1}(1-\epsilon)\delta \leq \frac{2}{\delta+1}(1-\frac{\delta-1}{2\delta})\delta =1 $.\\
Furthermore
with $\epsilon\leq \left( \frac{\delta-1}{2\delta}\right)^{\frac{1}{k}} $
we have $\left(1-\epsilon^k\right)\delta \geq \left(1-\frac{\delta-1}{2\delta}\right)\delta= \frac{\delta+1}{2}=\lambda $.\\
Then
\begin{equation}\label{probability4}
\lambda \leq \left(1-\epsilon^k\right)\delta.
\end{equation}
Clearly $\Pr\left(X_j=1\right)= 1-\left( 1-\lambda_0 x^*_j \right)^{k}$.

 Define the function $f$ by $ f(x)=\frac{1- (1-x)^{k}}{x}$.\\
  $f$ is strictly decreasing
 on $(0,1]$. Therefore, \\
 $\  \frac{1- \left(1-\lambda_0 x^*_j\right)^{k}}{\lambda_0 x^*_j} \leq \frac{1- \left(1-(1-\epsilon)\right)^{k}}{1-\epsilon}= \frac{1- \epsilon^{k}}{1-\epsilon}$.
 
 It follows that $ \Pr\left(X_j=1\right)  \leq \frac{1- \epsilon^{k}}{1-\epsilon} \cdot\lambda_0 x^*_j $.

 Then 
 \begin{equation}\label{probability5}
\Pr\left(X_j=1\right) \leq \left(1-\epsilon^k\right)\delta x^*_j.
\end{equation}

By using the LP relaxation and the definition of the sets
$C_1$ and $C_2$, and since $\lambda x^*_j\geq 1$ for all ${E_j\in C_1}$, we get
\begin{eqnarray*}
\mathbb{E}(X)&=& |C_1|+
\sum_{E_j\in C_2}\Pr\left(X_j=1\right)
+\sum_{E_j\in C_3}\Pr\left(X_j=1\right)\\
&\overset{(\ref{probability5})}{\leq} &\sum_{E_j\in C_1}\lambda x^*_j + \sum_{E_j\in C_2}(1-\epsilon^k)\delta x^*_j+ \sum_{E_j\in C_3}(1-\epsilon^k)\delta x^*_j\\
&\overset{(\ref{probability4})}{\leq} & (1-\epsilon^k)\delta\sum_{E_j\in \mathcal{E}} x^*_j \\
&\leq &  (1-\epsilon^k)\delta \mathrm{Opt}^{*} .
\end{eqnarray*}

%\smartqed\qed
(iv)
 Let us consider $\tilde{\mathcal{H}} $
   the subhypergraph induced by $C_1$ in which degree equality gives\\
  \begin{equation*}\label{double-countingSsuperieur}
  \sum_{i\in V}d(i)=\sum_{E_j\in C_1}|E_j|.
\end{equation*}\\
As the minimum vertex degree in the subhypergraph 
$\tilde{\mathcal{H}} $ is $b-1$ with $b:=\min_{i\in [n]}b_{i}$,
 we have
$$ (b-1)n\leq \sum_{i\in V}d(i)=\sum_{E\in C_1}|E_j|\leq \ell |C_1|.$$
Therefore
\begin{equation*}
 \frac{(b-1)n}{\ell}\leq | C_1|.
\end{equation*}
 Since  $|C_1|< \alpha \cdot \mathrm{Opt}^{*}$ we obtain
\begin{eqnarray*}
\frac{(b-1)n}{\alpha \ell} < \mathrm{Opt}^{*}.
\end{eqnarray*}
\end{proof}
\begin{theorem}\label{Maintheorem}
Let $\mathcal{H}$ be a hypergraph with fixed maximum vertex degree $\Delta$ and maximum edge size $\ell$. Let $\alpha=\frac{(b-1)\delta \epsilon^k }{6 \ell}\times \exp\left(a_{k,\epsilon}\right)$ with $ a_{k,\epsilon}=\frac{ k(1-\epsilon) +(\delta-1)(1-\epsilon^k)}{2}$ and $\epsilon \in \left[\frac{\delta-1}{2\delta},\left(\frac{\delta-1}{2\delta}\right)^{\frac{1}{k}}\right]$ as in Algorithm~\ref{alg:msetcover1}. The Algorithm~\ref{alg:msetcover1} returns a set multicover $C$ such that 
$$|C| < \max\left\{ \left(1 - \frac{1}{2}\left(1-\epsilon\right)\epsilon^k \right)\delta, \left(1- \frac{(b-1)\epsilon^k}{18 \ell}\times \exp\left(a_{k,\epsilon}\right) \right)\delta \right\}\cdot \mathrm{Opt}^{*} $$
with probability greater than $0.65$.
\end{theorem}
\vskip 0.2cm\noindent
\begin{proof}
Let $\mathcal{C}$ be the event that the inequality
$|C|<\left(1 - \frac{1}{2}\left(1-\epsilon\right)\epsilon^k \right)\delta\cdot \mathrm{Opt}^{*}$ is satisfied.
It suffices to prove that event $\mathcal{C}$ holds with the given probability in the case
 $|C_1|< \alpha \cdot \mathrm{Opt}^{*}$ since the opposite case is 
  discussed in Theorem~\ref{theoremcase1}.
 For this purpose, we estimate both the concentration of $X$ and $Y$ around their expectation. 
 Choose $t=2n\exp\left(-a_{k,\epsilon}\right)$ and consider $\mathcal{A}$ the event
$Y \leq n(1- \exp\left(-2a_{k,\epsilon}\right))-t$.\\
This involves
\begin{eqnarray*}
 n\exp\left(-2a_{k,\epsilon}\right) +t &=& n\exp\left(-2a_{k,\epsilon}\right)+ 2n\exp\left(-a_{k,\epsilon}\right)\\
 &\leq & 3n\exp\left(-a_{k,\epsilon}\right) \\
&= &  \frac{n(b-1)}{\ell}\cdot \frac{6\ell}{(b-1)\delta \epsilon^k \exp\left(a_{k,\epsilon}\right) }\cdot \frac{1}{2}\epsilon^k \delta\\
&= &  \frac{n(b-1)}{\alpha \ell} \cdot \frac{1}{2}\epsilon^k \delta\\
&\overset{\textrm{Lem }~\ref{lemma:random} (iv)}{\leq} &  \frac{1}{2}\epsilon^k \delta \cdot \mathrm{Opt}^{*}.
\end{eqnarray*}
And by Lemma~\ref{lemma:random}(ii) we have $\frac{t^2}{{\rm Var}(Y)}\geq \frac{4n^2\exp\left(-2a_{k,\epsilon}\right)}{2n^2\exp\left(-2a_{k,\epsilon}\right)}=2$. \\
\\
Therefore
\begin{eqnarray*}
\Pr\left(\mathcal{A} \right)
& \leq & \Pr\left(Y \leq \mathbb{E}(Y)-t \right)\\
& \overset{\textrm{Th }~\ref{Che-Can}}{\leq} & \frac{{\rm Var}(Y)}{ {\rm Var}(Y)+t^2}\\
& = & \frac{1}{ 1+\frac{t^2}{{\rm Var}(Y)}}\\
 & \leq & \frac{1}{3}.
\end{eqnarray*}
Consider now $\mathcal{B}$ the event
$X \geq \left(1-\left(1 - \frac{1}{2}\epsilon \right)\epsilon^k\right)\delta \mathrm{Opt}^{*}$.
Our basic assumption is to consider $\delta$, $k$ and $\epsilon $ constants, and we can certainly assume that
 $n\geq  \frac{ 16 \exp\left(a_{k,\epsilon}\right)}{\epsilon^{k+2}  } $,
since otherwise we obtain an optimal solution for the \sbmultcov problem
in polynomial time.\\
Choosing $\beta=\frac{1}{2}\epsilon^{k+1}$ we have
 \begin{eqnarray*}
(1+\beta)(1-\epsilon^k)
& = & 1-\epsilon^k + \frac{1}{2}\epsilon^{k+1}-\frac{1}{2}\epsilon^{2k+1} \\
& = & 1-\epsilon^k \left(1 - \frac{1}{2}\epsilon +\frac{1}{2}\epsilon^{k+1}\right)\\
& \leq & 1-\left(1 - \frac{1}{2}\epsilon \right)\epsilon^k.\\
\end{eqnarray*}

Note that $\epsilon \in \left[\frac{\delta-1}{2\delta},\left(\frac{\delta-1}{2\delta}\right)^{\frac{1}{k}}\right]$
therewith 
$1-\epsilon^k \geq 1- \frac{\delta-1}{2\delta} = \frac{\delta+1}{2\delta}>\frac{1}{2}$.\\
 We thus get
\begin{eqnarray*}
\Pr\left( \mathcal{B} \right)
& \leq &
\Pr\left(X \geq (1 + \beta)\cdot(1-\epsilon^k)\delta \mathrm{Opt}^{*} \right)\\
&\overset{\textrm{Th }~\ref{Doerr}}{\leq}&
\exp\left(- \frac{\beta^2 (1-\epsilon^k)\delta \mathrm{Opt}^{*}}{3}\right)\\
&\overset{ \textrm{Lem.}\ref{lemma:random}(iv)}{\leq}&
\exp\left(- \frac{\epsilon^{2k+2} (1-\epsilon^k)\delta n(b-1)}{12\ell}\cdot
\frac{6\ell}{(b-1)\delta \epsilon^k \times \exp\left(a_{k,\epsilon}\right)}\right)\\
&\leq &
\exp\left(- \frac{\epsilon^{k+2}  (1-\epsilon^k) n}{ 2 \exp\left(a_{k,\epsilon}\right)}\right)\\
&\leq &
\exp\left(- \frac{\epsilon^{k+2}  n}{ 4 \exp\left(a_{k,\epsilon}\right)}\right)\\
&\leq & \exp\left(-4\right).
\end{eqnarray*}
Therefore it holds that
\begin{eqnarray}\label{Intersection}
\Pr\left( \overline{\mathcal{A}} \cap \overline{\mathcal{B}} \right) 
&\geq &  1-\left(\frac{1}{3}+\exp\left(-4\right)\right),
\end{eqnarray}
where $ \overline{\mathcal{A}} $ and $\overline{\mathcal{B}}$ denote the complement events of $\mathcal{A}$ and $\mathcal{B}$ respectively.
We conclude that
\begin{eqnarray*}
\Pr\left(\mathcal{C}\right) & = &
\Pr\left(|C|\leq \left(1 - \frac{1}{2}\left(1-\epsilon\right)\epsilon^k  \right) \delta\mathrm{Opt}^{*}\right)\\
& = & \Pr\left(|C|\leq \left(1-\left(1 - \frac{1}{2}\epsilon \right)\epsilon^k+\frac{1}{2}\epsilon^k \right) \delta\mathrm{Opt}^{*}\right)\\
&\overset{(\ref{expection})}{\geq} & \Pr\left(X+n-Y\leq\left(1-\left(1 - \frac{1}{2}\epsilon \right)\epsilon^k+\frac{1}{2}\epsilon^k \right) \delta\mathrm{Opt}^{*}\right)\\
&\geq & \Pr\left(X \leq \left( 1-\left(1 - \frac{1}{2}\epsilon \right)\epsilon^k \right)\delta  \mathrm{Opt}^{*} \text{\ and\ }
n-Y \leq \frac{1}{2}\epsilon^k \delta \mathrm{Opt}^{*} \right)\\
&\geq & \Pr\left(X \leq \left( 1-\left(1 - \frac{1}{2}\epsilon \right)\epsilon^k \right)\delta \mathrm{Opt}^{*} \text{\ and\ }
Y \geq n-\frac{1}{2}\epsilon^k \delta \mathrm{Opt}^{*} \right)\\
&\geq & \Pr\left(X \leq \left( 1-\left(1 - \frac{1}{2}\epsilon \right)\epsilon^k \right)\delta \mathrm{Opt}^{*} \text{\ and\ }
Y \geq n-n\exp\left(-2a_{k,\epsilon}\right)-t \right)\\
&\geq & \Pr\left(X  \leq \left( 1-\left(1 - \frac{1}{2}\epsilon \right)\epsilon^k \right)\delta\mathrm{Opt}^{*} \text{\ and\ }
Y \geq n(1-\exp\left(-2a_{k,\epsilon}\right))-t  \right)\\
& \overset{(\ref{Intersection} )}{\geq} & 1-\left(\frac{1}{3}+\exp\left(-4\right)\right)\\
&\geq & 0.65.
\end{eqnarray*}
\end{proof}
\hfill$\Box$\\
 {\bf Remark 2.} The proof above gives for $k=2$ and $\epsilon=\frac{1}{2}$ an approximation ratio of 
$ \max\left\{ \frac{15}{16}\delta, \left(1- \frac{(b-1)\exp\left(\frac{ 3\delta+1}{8}\right)}{72 \ell}  \right)\delta\right\} $. Note that $ \frac{\delta-1}{2\delta}<\frac{1}{2}<  \left( \frac{\delta-1}{2\delta}\right)^{\frac{1}{2}}$ therewith the condition of Theorem~\ref{Maintheorem} on $\epsilon$ is satisfied. 
 
 As mentioned above our performance guaranty improves over the ratio presented by Srivastav et al~\cite{EMS16}, and this without restriction on the parameter $\ell$.
 Namely, for $\delta \geq 13$ we have
 \begin{eqnarray*}
11(\delta-1) < \exp\left(\frac{ 3\delta+1}{8}\right)
 & \Rightarrow &
 \frac{11(\delta-1)}{72\ell}< \frac{\exp\left(\frac{ 3\delta+1}{8}\right)}{72\ell}\\
  &\overset{b-1\geq 1}{ \Rightarrow }&
 \frac{11(\Delta-b)}{72\ell}< \frac{(b-1)\exp\left(\frac{ 3\delta+1}{8}\right)}{72\ell}\\
  & \Rightarrow &
 \left(1-\frac{(b-1)\exp\left(\frac{ 3\delta+1}{8}\right)}{72\ell}\right) \delta
<
\left(1-\frac{11(\Delta-b)}{72\ell}\right) \delta. 
 \end{eqnarray*}

\section{Lower Bound}
 One of the features of the proof is the duality of hypergraphs. In dual hypergraphs, vertices and edges just swap the roles. So the  \sbmultcov problem in dual hypergraphs becomes as follows: find a minimum cardinality set $C\subseteq V $ such that for every $E\in \mathcal{E}$ it holds $|E\cap C|\geq b$. This problem is known as the $b$-vertex cover problem and we have that the  \sbmultcov problem in $\Delta$-regular hypergraphs is equivalent to the $b$-vertex cover problem in $\Delta$-uniform hypergraphs.
\begin{theorem}\label{SC} 
Let $\epsilon >0$, $\Delta$ and $\bf b\in\mathbb{N}_0^n$ be given and $b=\min_i b_i$.
Then, it is NP-hard to approximate the \sbmultcov problem on $\Delta$-regular hypergraphs within a factor of $\Delta-b-\epsilon$.\@
\end{theorem}
\emph{Proof}.
 Assume, for a contradiction, that the theorem is false.  
Then there exists an algorithm $\mathcal{A}$ that returns a $(\Delta-b-\epsilon)$-approximation in polynomial time for the $b$-vertex cover problem on $\Delta$-uniform hypergraphs.
 
 We give a reduction of  the minimum vertex cover problem on $\Delta$-uniform hypergraphs to the $b$-vertex cover problem on $\Delta+b-1$-uniform hypergraphs.
 
 Let $\tilde{\mathcal{H}}=(V,\mathcal{E})$ be a $\Delta$-uniform hypergraph 
 and let $\alpha=\frac{2}{\epsilon}(\Delta-1-\epsilon)(b-1)$.\\
  Now we consider the following algorithm:
 \begin{itemize}
 \item[1.] Consider all subsets $T\subseteq V$ with $|T|\leq\alpha$. Check if any of these subsets is a vertex cover in $\tilde{\mathcal{H}}$. If it's the case then return the smallest one of them, else go to step 2.
 \item[2.] Add $b-1$ vertices $v_{1},\ldots,v_{b-1}$ to $V$. Define for every hyper-edge $E$ a new edge $E^{\ast}:=E\cup \{v_{1},\ldots,v_{b-1}\}$ and the set $\mathcal{E}^{\ast}:=\{E^{\ast}|E\in \mathcal{E}\}$. Finally set $\mathcal{H}=(V\cup\{v_{1},\ldots,v_{b-1}\},\mathcal{E}^{\ast})$. We execute $\mathcal{A}$ on $\mathcal{H}$. Return $T:=\mathcal{A}(\mathcal{H})\cap V$.
 \end{itemize}
\paragraph{\bf Claim} The algorithm given above returns a vertex cover in $\tilde{\mathcal{H}}$ in polynomial-time with an approximation ratio of $\Delta-1-\frac{\epsilon}{2}$.

\emph{Proof}.\\ 
{\bf Correctness and approximation ratio.} 
If $T$ is selected by the algorithm in step $1$ then $T$ is an optimal vertex cover in $\tilde{\mathcal{H}}$.\\
If $T$ is selected by the algorithm in step $2$ then 
$\mathcal{A}(\mathcal{H})= T \cup K$ 
for some $K\subset \{v_{1},\ldots,v_{b-1}\}$. Note that $T$ and $K$ are disjoint sets. Consider an edge $E\in\mathcal{E}$.
Because $\mathcal{A}(\mathcal{H})$ is a $b$-vertex cover in $\mathcal{H}$, we have $|\mathcal{A}(\mathcal{H})\cap E^{\ast}|= |T\cap E^{\ast}|+ |K\cap E^{\ast}|\geq b$. Since $T\cap E^{\ast}=T\cap E$ and $|K\cap E^{\ast}|\leq b-1$, it follows that $|T\cap E|\geq 1$. Hence $T$ is a vertex cover in $\tilde{\mathcal{H}}$.\\
Now, let $C$ and
$C'$ denote a minimum vertex cover in $\tilde{\mathcal{H}}$ and a minimum $b$-vertex cover in $\mathcal{H}$, respectively.
 Since $D^{'}:= C\cup \{v_{1},\ldots,v_{b-1}\}$ is a feasible $b$-vertex cover in $\mathcal{H}$, it holds that $|C'|\leq |C|+b-1$.
On the other hand, it is clear that $\mathcal{H}$ is a $\left(\Delta+b-1\right)$-uniform hypergraph, and by the assumption we get
 \begin{align*}
  |\mathcal{A}(\mathcal{H})|
 &\leq \left((\Delta+b-1)-b-\epsilon\right)|C'|\\
  &\leq (\Delta-1-\epsilon)|C|+(\Delta-1-\epsilon)(b-1)=(\Delta-1-\epsilon)|C|+\frac{\epsilon}{2}\alpha\\
  & \overset{|C|\geq \alpha} {\leq} (\Delta-1-\frac{\epsilon}{2})|C|.
\end{align*} 
Since $|T|= |\mathcal{A}(\mathcal{H})\cap V|\leq  |\mathcal{A}(\mathcal{H})|$, it follows that
$|T|\leq (\Delta-1-\frac{\epsilon}{2})|C|$.

{\bf Running time.} In step 1 we test at most $n^{\alpha}$ sets of vertices to be a vertex cover in $\tilde{\mathcal{H}}$. Since $\alpha=\frac{2}{\epsilon}(\Delta-1-\epsilon)(b-1)$ is a constant, the running time in this step is polynomial. In step 2 we add a constant number of vertices to $V$ and execute the algorithm $\mathcal{A}$. Hence the algorithm runs in polynomial time in both steps.

With Claim 1 there is a factor $\Delta-1-\frac{\epsilon}{2}$ approximation algorithm for the minimum vertex cover problem on $\Delta$-uniform hypergraphs, which contradicts the statement of Theorem~\ref{Dinur}.
\hfill$\Box$
\section{The $\frac{\ln_{2}(n+1)}{2b}$-Integrality Gap}
The integrality gap for \sbmultcov problem is defined as the supremum of the ratio $\frac{{\rm Opt}_{{\bf b}}(\mathcal{H})}{{\rm Opt}^{*}_{{\bf b}}(\mathcal{H}) }$ over all instances $\cal{H}$ of the problem.
In this section we give a slight modification of the proof presented in~\cite{Vazirani01} for the integrality gap. We present in the following a specific class of instances of the \sbmultcov problem, where ${\bf b}:=(b,\ldots,b)\in \mathbb{N}^{n}$ for which the integrality gap is at least $\frac{\ln_{2}(n+1)}{2b}$.
\begin{theorem}\label{IG}
let ${\bf b}:=(b,\ldots,b)\in \mathbb{N}^{n}$.
% and let $\mathcal{H}=(V,\mathcal{E})$ be a $\Delta$-regular hypergraph.
The integrality gap of the \sbmultcov problem is at least $\frac{\ln_{2}(n+1)}{2b}$.
\end{theorem}
Define $ V=F_2^k\backslash\{0\}$ as the set
of all $k$-dimensional non-zero vectors with component values of $\mathbb{Z}_2 =\{ 0,1\}$ for a fixed integer $k$  and we define ${\cal E}$ as a collection of the sets $E_v=\{u\in  V: <v,u>\equiv 1[2]\}$
 for each $v\in V$, where $ <.\,,.> $ is the usual dot product in $ V$.

We remark that each element $v\in  V$ is contained in exactly half of the sets of ${\cal E}$ therewith the hypergraph  ${\cal H}=(V,{\cal E})$ is regular and  $n=| V|=2^k-1$.

\begin{lemma}\label{fraction}
Let ${\cal H}=(V,{\cal E})$ the hypergraph defined and  ${\bf b}\in \mathbb{N}_{\geq 1}^{n}$. 
It holds that
 the vector $x=(\frac{2b}{|{\cal E}|},\ldots,\frac{2b}{|{\cal E}|})$ is a feasible solution for  LP($\Delta,\,{\bf b}$).

\end{lemma}

\emph{Proof}. 
It is clear that $x=(\frac{2b}{|{\cal E}|},\ldots,\frac{2b}{|{\cal E}|})$ is a feasible solution for lP$(\Delta,{\bf b})$, namely since $\mathcal{H}$ is regular  with $\Delta=\frac{|{\cal E}|}{2}$ we have for every $i\in \{1,\ldots,n\}$
\begin{equation*}
\sum_{E\in \Gamma(v_i)}\frac{2b}{|{\cal E}|}=\frac{2b}{|{\cal E}|}\cdot \Delta= \frac{2b}{|{\cal E}|}\cdot \frac{|{\cal E}|}{2}\geq b
\end{equation*}
therewith $\mathrm{Opt}^{*}\leq 2b$.\hfill$\Box$\\

\begin{lemma}\label{optimal} The optimal integral solution to the previous LP formulation of the \sbmultcov problem requires
at least $k$ sets.
\end{lemma}
\emph{Proof}.
Let  
$\{ E_{v_1}, E_{v_2}\ldots E_{v_t} \}$  a collection of sets such that ${\bigcup}_{i\in [t]}E_{v_i}= F_2^k\backslash\{0\}$. 
This implies that the intersection of their complements contains exactly the zero vector,
i.e., ${\bigcap}_{i\in [t]} E_{v_i}^{C}=\{0\}$. 
It follows that $0$ is the only solution in $F_2^k$ of the system
\[
<x,v_i>\equiv 0[2],\quad\forall i\in [t]
\]
Then it holds that $t\geq k$, since the dimension of $F_2^k$ is $k$ while the number of the equations in the system is $t$. From this we conclude $\mathrm{Opt}\geq \ln_{2}(n+1)$.
\hfill$\Box$\\
\smallskip

{\bf Proof of Theorem~\ref{IG}.} Theorem~\ref{IG} follows from Lemma~\ref{fraction} and Lemma~\ref{optimal}. \hfill$\Box$\\
\smallskip
\section{Future Work}
We believe now that the conjecture of Peleg et al.\ holds in the general setting. Hence proving the trueness of the conjecture remains a big challenge for our future works.


\begin{thebibliography}{aaa}
\bibitem{Ba01}
R. Bar-Yehuda. \emph{Using Homogeneous Weights for Approximating the Partial Cover Problem.}
Journal of Algorithms, 39(2):137--144, 2001.

%\bibitem{B73}
%C. Berge. \emph{Graphs and Hypergraphs}. North-Holland 1973.

\bibitem{DinKhotRegev}
I. Dinur, V. Guruswami, S. Khot, O. Regev,
\emph{ A new multilayered PCP and the hardness of hypergraph vertex cover}, SIAM J. Comput. 34 (5)
 1129--1146, 2005.

\bibitem{Doerr}
 B. Doerr, F. (Eds.)  Neumann.
 \emph{ Theory of Evolutionary Computation: Recent Developments in Discrete Optimization}. Springer Nature 2019.
 
 \bibitem{Duh97}
R. Duh, M. Fürer.
\emph{ Approximating k-set cover by semi-local optimization.} in: Proc. 29th Annual Symposium
on Theory on Computing, May, pp. 256--264, 1997.

\bibitem{EFS14}
M. El Ouali, H. Fohlin, A. Srivastav.
\emph{An approximation algorithm for the partial vertex cover problem in hypergraphs}.
Journal of Combinatorial Optimization, 31(2): 846--864, 2016.

\bibitem{EFS14a}
M. El Ouali, H. Fohlin, A. Srivastav.
\emph{A Randomised approximation algorithm for the hitting set problem}.
Theoretical Computer Science 555: 23--34, 2014.

\bibitem{EMS16}
M. El Ouali, P. Munstermann, A. Srivastav. \emph{Randomized approximation for set multicover in hypergraphs}. Algorithmica, 74(2): 574--588, 2016.

 \bibitem{Feige}
 U. Feige. 
\emph{ A threshold of $\ln n$ for approximating
 set cover.} J. ACM, 45(4):634--652, 1998.

 \bibitem{Fujito}
 T. Fujito, H. Kurahashi.
\emph{ A Better-Than-Greedy Algorithm for k-Set Multicover.}
 In: 3rd International Workshop on Approximation and Online Algorithms, pp. 176--189, 2006. 

\bibitem{GaKhSr01}
R. Gandhi, S. Khuller and A. Srinivasan. \emph{Approximation Algorithms 
for Partial Covering Problems.} Journal of Algorithms, 53(1):55--84, 2004.

\bibitem{HH86}
N.G. Hall, D.S. Hochbaum.
\emph{A Fast Approximation Algorithm for the Multicovering Problem.}
Discrete Applied Mathematics, 15:35--40, 1986.

\bibitem{Hoch82}
D.S. Hochbaum.
\emph{Approximation Algorithms for the Set Covering and Vertex Cover Problems.}
SIAM J. Comput, 11(3):555--556, August 1982.

 \bibitem{Johnson74}
D. S. Johnson. \emph{Approximation Algorithms for Combinatorial Problems.}
Journal of Computer and System Sciences, 9:256--278, 1974.

\bibitem{Karp}
R. KARP, \emph{Reducibility among combinatorial problems.} In R.E. Miller and
J.W. Thatcher, editors, Complexity of Computer Computations, pp. 85--103.
Plenum Press, New York, NY, 1972.

\bibitem{KR08}
S. Khot and O. Regev. \emph{Vertex Cover Might be Hard to Approximate to Within 2-epsilon.}
Journal of Computer and System Sciences, 74(3):335--349, 2008.

\bibitem{Koufogiannakis}
C. Koufogiannakis, N.E. Young. 
  \emph{Greedy $\Delta$-approximation algorithm for covering with arbitrary constraints and submodular cost}. Algorithmica, 66(1), 113--152, 2013.

\bibitem{Lovasz94}
L. Lov\'asz. \emph{On the Ratio of Optimal Integral and Fractional Covers.}
Discrete Mathematics, 13(4):383--390, 1975.


\bibitem{MR}
R. Motwani, P. Raghavan. \emph{Randomized Algorithms}. Cambridge University Press 1995.

\bibitem{PSW97}
D. Peleg, G. Schechtman, A. Wool.
\emph{Randomized Approximation of Bounded Multicovering Problems.}
Algorithmica, 18(1):44--66, 1997.

\bibitem{PSW93}
D. Peleg, G. Schechtman, A. Wool.  
\emph{Approximating bounded 0-1 integer linear programs.} In
Proc.\ 2nd Israel Symp.\ on Theory of Computing Systems, pp.\ 69--77, Netanya, 1993. 

\bibitem{RajVaz}
S. Rajagopalan,  V. V. Vazirani. 
\emph{Primal-dual RNC approximation algorithms for set cover and covering integer programs}. SIAM J. Comput., 28(2), 525--540, 1998.

\bibitem{Vazirani01}
 V. V. Vazirani.
 \emph{Approximation Algorithms}, pp.\ 108--112, Springer 2001.
\end{thebibliography}
\end{document}